\documentclass{amsart}

\usepackage{amsmath}
\usepackage{amssymb}
\usepackage{amsthm}
\usepackage{amsfonts}
\usepackage{calligra}
\usepackage{appendix}
\usepackage{graphicx}
\usepackage{enumerate}
\usepackage{mathrsfs}
\usepackage[english]{babel}
\usepackage{mathtools}
\usepackage{hyperref}

\def\R{\mathbb R}

\def\N{\mathbb N}

\def\T{\mathbb T}
\def\A{\mathbf{A}}

\def\E{\mathbf{E}}
\def\u{\mathbf{u}}

\def\B{\mathbf{B}}
\def\x{\mathbf{x}}
\def\d{\,\textup{d}}

\def\tp{\textup}

\newcommand{\abs}[1]{\left\lvert#1\right\rvert}

\newcommand{\defeq}{\mathrel{:\mkern-0.25mu=}}
\newcommand{\eqdef}{\mathrel{=\mkern-0.25mu:}}

\newtheorem{theorem}{Theorem}
\newtheorem{lemma}[theorem]{Lemma}
\newtheorem{proposition}[theorem]{Proposition}

\theoremstyle{definition}

\theoremstyle{remark}

\title{Relaxation of the kinematic dynamo equations}

\begin{document}

\author{Lauri Hitruhin}
\address{Department of Mathematics and Statistics \\ University of Helsinki, P.O. Box 68, 00014 Helsingin yliopisto, Finland}
\email{lauri.hitruhin@helsinki.fi}
\thanks{L.H. was supported by the Academy of Finland projects \#332671 and SA-1346562. S.L. was supported by the ERC Advanced Grant 834728.}

\author{Sauli Lindberg}
\address{Department of Mathematics and Statistics \\ University of Helsinki, P.O. Box 68, 00014 Helsingin yliopisto, Finland}
\email{sauli.lindberg@helsinki.fi}

\date{}

\begin{abstract}
We compute the exact relaxation and $\Lambda$-convex hull of the kinematic dynamo equations and show that they coincide. We also find the relaxation in the stationary case.
\end{abstract}

\maketitle

\section{Introduction}
\emph{Dynamo theory} studies the mechanisms by which electrically conducting fluids generate and sustain the magnetic fields of celestial bodies \cite{AK, CG}. In typical astrophysical applications, the magnetic Reynolds number is very large, and so one often studies the idealised limit of a perfect conductor ~\cite[p. 9]{CG}. We also assume perfect conductivity in this paper. Dynamo theory is divided into two parts.

In the simpler \emph{kinematic dynamo model}, the interaction of the magnetic field and the fluid is described via the induction equation and Gauss's law for magnetism,
\begin{equation} \label{e: Kinematic dynamo equations}
\partial_t \B + \nabla \times (\B \times \u) = 0, \qquad \nabla \cdot \B = 0.
\end{equation}
In \eqref{e: Kinematic dynamo equations}, $\B$ is the magnetic field and $\u$ is the fluid velocity. Sometimes, the incompressibility condition $\nabla \cdot \u = 0$ is considered as a part of the kinematic model. In the \emph{nonlinear dynamo model}, \eqref{e: Kinematic dynamo equations} is extended to the full ideal incompressible MHD (magnetohydrodynamic) equations by adding the Cauchy Momentum equation with Lorentz force and incompressibility,
\begin{equation}
\partial_t \u + (\u \cdot \nabla) \u = -\nabla p + (\nabla \times \B) \times \B, \qquad \nabla \cdot \u = 0. \label{e:Cauchy momentum equation}
\end{equation}
The kinematic dynamo equations \eqref{e: Kinematic dynamo equations} also allow one to study very general $\u$ regardless of constitutive laws, compressibility, viscosity, external forcing and so on~\cite{Eyink2}. Besides ideal or viscous MHD, one can e.g. set $\u$ to be the electron fluid velocity in Hall MHD~\cite{Eyink} or the velocity field in Moffatt's magnetic relaxation equations~\cite{BFV,Moffatt} or magneto-friction~\cite{Yeates}.

From a mathematical viewpoint, kinematic dynamo equations are also a useful toy model for MHD. While the highly difficult analysis of the back-propagation of the Lorentz force on the fluid is avoided, some of the salient features of the MHD model are still retained. In particular, one of the classical conserved quantities of MHD, magnetic helicity $\mathcal{H}_M(t) \defeq \int_{\Omega} \A(x,t) \cdot \B(x,t) \d x$, is conserved by weak solutions $\B,\u \in L^3_{t,x}$ of \eqref{e: Kinematic dynamo equations} e.g. under periodic boundary conditions on $\Omega = \T^3$ (see~\cite{FLSsubmitted,FLS22} for more information). Magnetic helicity constrains the dynamics, in particular via \emph{Arnold's inequality} $\int_{\Omega} \abs{\B(x,t)}^2 \d x \geq c\int_{\Omega} \A(x,t) \cdot \B(x,t) \d x$ which keeps magnetic energy bounded away from zero if the initial field $\B_0$ has non-zero magnetic helicity. By~\cite[Theorem 2.2]{FLS}, conservation even extends to a macroscopically averaged version, the so-called \emph{relaxation} of \eqref{e: Kinematic dynamo equations}. In the relaxation, solutions of \eqref{e: Kinematic dynamo equations} are replaced by weak limits of solutions, as we next make precise.

In the \textit{Tartar framework}~\cite{Tartar}, the kinematic dynamo equations are decoupled into the conservation laws
\begin{equation} \label{eq:relaxedMHD}
\nabla \cdot \B = 0, \qquad \partial_t \B + \nabla \times \E = 0
\end{equation}
(\textit{Gauss's law for magnetism} and \textit{Faraday's law of induction}) and the constitutive law $\E = \B \times \u$ (\textit{ideal Ohm's law}) which is codified into the condition that $(\B,\u,\E)$ takes values in the \emph{constraint set} $K = \{(B,u,E): E = B \times u\}$. One can also encode more information into the constraint set; one can e.g. prescribe kinetic and magnetic energy densities by letting $r,s > 0$ and considering the normalised set \[K_{r,s} = \{(B,u,E) \colon E = B \times u, \, \abs{u}=r, \, \abs{B} = s\}.\]
The \textit{relaxation of $K_{r,s}$} can be defined as the smallest set $\widetilde{K_{r,s}}$ such that whenever solutions of \eqref{eq:relaxedMHD} take values in $K_{r,s}$, their weak limits take values in $\widetilde{K_{r,s}}$. Another, potentially larger (but in our case equal) variant is the smallest superset of  $K_{r,s}$ that is closed under weak convergence for solutions of \eqref{eq:relaxedMHD}, essentially following Tartar~\cite{Tartar}. The relaxation describes macroscopic averages of solutions of \eqref{e: Kinematic dynamo equations}~\cite{Tartar}.

We discuss the motivation for studying the relaxation of \eqref{e: Kinematic dynamo equations}. The bare fields $\B$ and $\u$ can be neither observed~\cite{Aluie} nor predicted~\cite{Davidson} effectively, and so most turbulence theories concentrate on various averaged quantities which are much better reproducible~\cite{Davidson}. Lax~\cite{Lax} has suggested weak limits as a possible deterministic substitute for ensemble averaging, and the idea was developed further in~\cite{BGK}. For a recent systematic review on the topic see~\cite{DLS21}. In the relaxation, the averaging is performed at all length scales simultaneously, whereas e.g. coarse-graining of MHD via spatial filtering~\cite{Aluie} leads to subscale stresses which depend on the filtering resolution and need to be modelled (e.g. by eddy viscosity models~\cite[p. 401]{Davidson}).

The relaxation gives geometric insights which remain hidden in conventional averaging methods. As an example, in the kinematic dynamo equations, the orthogonality $\B \cdot \E = 0$ extends to the relaxation even though the constitutive law $\E = \B \times \u$ breaks down; this plays a key role in showing that $L^3_{t,x}$ is the optimal integrability under which magnetic helicity is conserved~\cite{FLSsubmitted}. It is, at any rate, a challenging problem to find suitable physically motivated selection criteria which pick a unique \textit{subsolution} (i.e. a triple $(\B,\u,\E)$ which solves \eqref{e: Kinematic dynamo equations} and take values in $\widetilde{K_{r,s}}$) under given initial/boundary (or other) conditions. For such a criterion in the case of the Muskat problem with flat interface ("maximal mixing") see~\cite{Mengual,Otto,Sze}.

The relaxation is also used in convex integration, a mathematical technique which produces a weak solution from a subsolution via an $h$-principle~\cite{DLS12}. The use of convex integration in fluid mechanics was pioneered by De Lellis and Sz\'ekelyhidi in~\cite{DLS}. By a refined quantitative version of the $h$-principle one can produce solutions whose averages are dictated by the subsolution (a rigorous statement was proved by Castro, Faraco and Mengual in~\cite{CFM}). Furthermore, instabilities in fluid dynamics, which result in a turbulent evolution intractable with the classical theory, have been successfully modelled by these techniques; see for example~\cite{CCF, CFM22} for Rayleigh-Taylor under IPM,~\cite{GK22, GKS} for Rayleigh-Taylor under inhomogeneous Euler,~\cite{GK, MS} for Kelvin-Helmoltz and the references contained therein.

Precise information on the relaxation is imporant in identifying the boundary/initial conditions under which convex integration can be run in the Tartar framework~\cite{Sze}. It is a natural meta-conjecture in fluid dynamics that the relaxation coincides with the \textit{$\Lambda$-convex hull} $K_{r,s}^\Lambda$, i.e., the set of points which cannot be separated from $K_{r,s}$ by a $\Lambda$-convex function; see \textsection \ref{s:Preliminaries} for the relevant definitions. We prove the conjecture in the case of the kinematic dynamo equations. Theorem \ref{Toinen päätulos} is also a step towards computing the exact relaxation of the full MHD system; a qualitative characterisation was given in~\cite[Theorem 6.7]{FLS}.

\begin{theorem}\label{Toinen päätulos}
The relaxation, first laminate and $\Lambda$-convex hull of the kinematic dynamo equations coincide and can be written as
\begin{align} \label{e: hull}
K_{r,s}^\Lambda
= & \{(B,u,E) \colon \abs{B} \leq r, \; \abs{u} \leq s, \nonumber  \\
& B \cdot E = 0, \; \abs{E-B \times u} \leq \sqrt{(r^2- \abs{B}^2) (s^2 - \abs{u}^2)}\}.
\end{align}
\end{theorem}

To close the introduction, we mention that \eqref{e: hull} also gives the relaxation of kinematic dynamo equations in the stationary and incompressible cases (as the wave cone remains unchanged). However, in the incompressible, stationary case, the wave cone (and, \emph{a fortiori}, the relaxation) are strictly smaller. We compute the corresponding relaxation in \textsection \ref{s:The relaxation of the stationary model under incompressibility}.

\section{Preliminaries} \label{s:Preliminaries}
We recall some definitions from the theory of differential inclusions and refer to~\cite{Kirchheim} for more information. The \textit{wave cone} $\Lambda$ consists of the directions $(\bar{B},\bar{u},\bar{E}) \in \R^9$ such that plane waves $(\B,\u,\E)(\x,t) \defeq h((\x,t) \cdot \xi) (\bar{B},\bar{u},\bar{E})$, $h \in C^\infty(\R)$, solve \eqref{eq:relaxedMHD} for some $\xi \in (\R^3 \setminus \{0\}) \times \R$. The wave cone of the (non-stationary as well as stationary) kinematic dynamo equations has been computed in~\cite[Lemma 5.2]{FLS}:

\begin{proposition} \label{prop:wavecone}
$\Lambda = \{(B,u,E) \colon B \cdot E = 0\}$.
\end{proposition}

Given any compact set $C \subset \R^3$, the \emph{laminates} $C^{k,\Lambda}$, $k \in \N_0$, of $C$ are defined as follows:
\begin{align*}
& C^{0,\Lambda} \defeq C, \\
& C^{k+1,\Lambda} \defeq \{(\lambda z_1 + (1-\lambda) z_2 \colon z_1,z_2 \in C^{k,\Lambda}, \, z_1-z_2 \in \Lambda, \, \lambda \in [0,1]\}. 
\end{align*}
The \emph{lamination convex hull} of $C$ is defined as
\[C^{lc,\Lambda} \defeq \cup_{k=0}^\infty C^{k,\Lambda}.\]
Recall also that a function $G \colon \R^3 \to \R$ is said to be \emph{$\Lambda$-convex} if $t \mapsto G(z_0 + tz) \colon \R \to \R$ is convex for every $z_0 \in \R \times \R^3 \times \R^3$ and $z \in \Lambda$. The \emph{$\Lambda$-convex hull} $C^\Lambda$ consists of points $z \in \R^n$ that cannot be separated from $C$ by a $\Lambda$-convex function. More precisely, $z \notin C^\Lambda$ if and only if there exists a $\Lambda$-convex function $G$ such that $G|_C \le 0$ but $G(z) > 0$. We have $C^\Lambda \supset C^{lc,\Lambda}$.

\section{The proof of Theorem \ref{Toinen päätulos}}
We prove Theorem \ref{Toinen päätulos} through a series of lemmas, starting by the inclusion "$\subset$" in \eqref{e: hull}. For this we need the following $\Lambda$-affine and convex, and hence also $\Lambda$-convex, functions.

\begin{lemma} \label{cor:Lambda-affine functions}
The function
\[G_1(B,u,E) \defeq B \cdot E\]
is $\Lambda$-affine and vanishes in $K_{r,s}^\Lambda$.
\end{lemma}

\begin{proof}
The $\Lambda$-affinity follows directly from Proposition \ref{prop:wavecone}. Since $G_1$ vanishes in $K_{r,s}$, it also vanishes in $K_{r,s}^\Lambda$.
\end{proof}
In Theorem \ref{Toinen päätulos}, the $\Lambda$-affine function $G_1$ gives the correct bounds for the directions of the vector $E - B \times u$, and to bound the length we need the following convex function.
\begin{lemma} \label{prop:useful Lambda-convex function}
The function 
\[G_2(B,u,E) \defeq \max_{0 \leq \alpha \leq 1} [\alpha (\abs{B}^2 - r^2) + (1-\alpha) (\abs{u}^2 - s^2) + 2 \sqrt{\alpha (1-\alpha)} \abs{B \times u - E}]\]
is convex and vanishes in $K_{r,s}$. (Thus, $G_2 \leq 0$ in $K_{r,s}^\Lambda$.)
\end{lemma}

\begin{proof}
Clearly $G_2|_{K_{r,s}} = 0$. Denoting $G_2 = \max_{0 \leq \alpha \leq 1} H_\alpha$, it then suffices to fix $\alpha \in [0,1]$ and show that $H_\alpha$ is convex. Fix $z_0 = (B_0,u_0,E_0)$ and $z = (B,u,E)$. Let $t \in \R$. Then
\begin{align*}
H_\alpha(z_0 + t z) - H_\alpha(z_0)
&\geq ct + (\alpha \abs{B}^2 + (1-\alpha) \abs{u}^2 - 2 \sqrt{\alpha (1-\alpha)} \abs{B \times u}) t^2 \geq c t
\end{align*}
for some constant $c \in \R$ which is independent of $t$, and so $H_\alpha$ is convex.
\end{proof}

The functions from Lemmas \ref{cor:Lambda-affine functions} and \ref{prop:useful Lambda-convex function} yield the inclusion "$\subset$" in \eqref{e: hull}:

\begin{lemma} \label{c: Upper bound on hull}
Every triple $(B,u,E) \in K_{r,s}^\Lambda$ satisfies
\begin{equation} \label{e:Upper bound for hull}
\abs{B} \leq r, \quad \abs{u} \leq s, \quad \abs{B \times u - E} \leq \sqrt{(r^2- \abs{B}^2) (s^2 - \abs{u}^2)}, \quad B \cdot E =0.
\end{equation}
\end{lemma}

\begin{proof}
Suppose $(B,u,E) \in K_{r,s}^\Lambda$. Now $\abs{B}^2 - r^2 \leq G_2(B,u,E) \leq 0$ and $\abs{u}^2 - s^2 \leq G_2(B,u,E) \leq 0$. For the third inequality note that if $\abs{B} = r$, then Proposition \ref{prop:useful Lambda-convex function} gives, for every $\alpha \in (0,1)$, $\abs{B \times u - E} \leq [\alpha/(1-\alpha)]^{1/2} (r^2-\abs{b}^2)$, so that $\abs{B \times u - E} = 0$. The case $\abs{u} = s$ is similar, and so we assume that $\abs{B} < r$ and $\abs{u} < s$. Choose $\alpha = (s^2-\abs{u}^2)/(s^2+r^2-\abs{B}^2-\abs{u}^2)$ in Proposition \ref{prop:useful Lambda-convex function}. Now $G_\alpha(B,u,E) \leq 0$, that is, $\abs{B \times u - E} \leq [(r^2- \abs{B}^2) (s^2 - \abs{u}^2)]^{1/2}$. The equality $B \cdot E = 0$ follows immediately from Proposition \ref{cor:Lambda-affine functions}.
\end{proof}

Note that by the convexity of $G_2$ and Jensen's inequality, the three inequalities in \eqref{e:Upper bound for hull} also extend to various other averaging processes (whereas the imporant equality $B \cdot E = 0$, which is at the core of magnetic helicity conservation~\cite{FLS22}, breaks down).

Our next task is to prove the inclusion "$\supset$" in \eqref{e: hull}, i.e., the sharpness of the upper bound \eqref{e:Upper bound for hull}. We start with the case $E = B \times u$.

\begin{lemma} \label{l:Simple combinations}
Suppose $\abs{B} \leq r$ and $\abs{u} \leq s$. Then $(B,u,B \times u) \in K_{r,s}^{1,\Lambda}$.
\end{lemma}

\begin{proof}
Choose $\bar{B},\bar{u} \in \tp{span}\{B,u\}^\perp$ with $\bar{B} \times \bar{u} = 0$, $ \abs{\bar{B}}^2 = r^2-\abs{B}^2$ and $\abs{\bar{u}}^2 = s^2 - \abs{u}^2$. Then
\begin{align*}
(B,u,B \times u)
&= \frac{1}{2} (B+\bar{B},u+\bar{u},(B+\bar{B}) \times (u+\bar{u})) \\
&+ \frac{1}{2}  (B-\bar{B},u-\bar{u},(B-\bar{B}) \times (u-\bar{u})) \eqdef \frac{z_1+z_2}{2}
\end{align*}
and $z_1-z_2 = (\bar{B}, \bar{u}, B \times \bar{u} + \bar{B} \times u) \in \Lambda$.
\end{proof}

The case $\abs{B} < r$, $\abs{u} < s$, $E - B \times u \neq 0$ requires more work. In Lemma \ref{l: lemma for main theorem} we give a somewhat explicit characterisation of the first laminate $K^{1,\Lambda}_{r,s}$; the proof of formula \eqref{e: hull} is then completed by showing that whenever $E \perp B$ is as indicated in Theorem \ref{Toinen päätulos}, we can write $E = B \times u + \bar{B} \times \bar{u}/[\abs{\bar{B}} \abs{\bar{u}}]$ for some $\bar{B},\bar{u}$ satisfying \eqref{e: condition on bar u and bar b}--\eqref{e: second condition on bar u and bar b}.

\begin{lemma} \label{l: lemma for main theorem}
Suppose $\abs{B} < r$ and $\abs{u} < s$. Then $(B,u,E) \in K^{1,\Lambda}_{r,s}$ if and only if there exist $\bar{B},\bar{u} \neq 0$ such that
\begin{align}
& E = B \times u + \sqrt{(r^2-\abs{B}^2)(s^2-\abs{u}^2)} \frac{\bar{B} \times \bar{u}}{\abs{\bar{B}} \abs{\bar{u}}} \label{e: Form of e}, \\
& \abs{\bar{B}}^2 = \frac{r^2-\abs{B}^2}{s^2-\abs{u}^2} \abs{\bar{u}}^2 = 4(r^2-\abs{B}^2 \sin^2 \alpha_{B,\bar{B}}), \label{e: condition on bar u and bar b}\\
& \abs{B} \cos \alpha_{B,\bar{B}} = \sqrt{\frac{r^2-\abs{B}^2}{s^2-\abs{u}^2}} \abs{u} \cos \alpha_{u,\bar{u}}, \label{e: third condition on bar u and bar b}\\
& B \cdot \bar{B} \times \bar{u} = 0, \label{e: second condition on bar u and bar b}
\end{align}
where we set $\abs{B} \sin \alpha_{B,\bar{B}} = \abs{B} \cos \alpha_{B,\bar{B}} = 0$ if $B = 0$ and similarly for $u$.
\end{lemma}

\begin{proof}
A general $\Lambda$-convex combination of two elements of $K_{r,s}$, with $\abs{B}<r$ and $\abs{u} < s$, is of the form
\begin{align}\label{Konveksi yhdiste ei-stationaarinen}
&\lambda V_1 + \mu V_2 \\
&= \lambda(B + \mu \bar{B}, u + \mu \bar{u}, (B + \mu \bar{B}) \times (u + \mu \bar{u})) \nonumber \\
&+ \mu (B - \lambda \bar{B}, u - \lambda \bar{u}, (B - \lambda \bar{B}) \times (u - \lambda \bar{u})) \nonumber \\
&= (B, u, B \times u + \lambda \mu \bar{B} \times \bar{u}), \nonumber
\end{align}
where $0 \leq \lambda \leq 1$, $\lambda + \mu = 1$,
\begin{equation}\label{e:Apuyht1}
\abs{B + \mu \bar{B}} = r, \quad \abs{B - \lambda \bar{B}} = r, \quad \abs{u + \mu \bar{u}} = s, \quad \abs{u - \lambda \bar{u}} = s, \quad B \cdot \bar{B} \times \bar{u} = 0.
\end{equation}
(In particular, $\bar{B},\bar{u} \neq 0$.)

We intend to show that \eqref{e:Apuyht1} is equivalent to \eqref{e: condition on bar u and bar b}--\eqref{e: second condition on bar u and bar b}. We first assume that \eqref{e:Apuyht1} holds and aim to prove \eqref{e: condition on bar u and bar b}. Squaring and subtracting on both sides of the first two equations in \eqref{e:Apuyht1} and moving terms (and recalling that $\lambda^2-\mu^2 = \lambda-\mu$),
\begin{equation}\label{e:Apuyht2}
2 B \cdot \bar{B} = (\lambda-\mu) \abs{\bar{B}}^2 = (2\lambda-1) \abs{\bar{B}}^2. \qquad \Longrightarrow \qquad \lambda = \frac{1}{2} + \frac{B \cdot \bar{B}}{\abs{\bar{B}}^2}.
\end{equation}
We next rewrite
\begin{align}\label{e:Apuyht5}
r^2
&= \abs{B + \mu \bar{B}}^2 = \abs{B}^2 - (2 B \cdot \bar{B} - \abs{\bar{B}}^2) \frac{B \cdot \bar{B}}{\abs{\bar{B}}^2} + \frac{(2 B \cdot \bar{B} - \abs{\bar{B}}^2)^2}{4 \abs{\bar{B}}^2} \nonumber \\
&= \abs{B}^2 + \frac{\abs{\bar{B}}^4 - 4 (B \cdot \bar{B})^2}{4 \abs{\bar{B}}^2}
\end{align}
and similarly $\lambda = 1/2 + u \cdot \bar{u}/\abs{\bar{u}}^2$ and $s^2 = \abs{u}^2 + [\abs{\bar{u}}^4-4(u \cdot \bar{u})^2]/[4\abs{\bar{u}}^2]$. We conclude that
\begin{equation} \label{e: conclusion on bar b and bar u}
\frac{r^2-\abs{B}^2}{\abs{\bar{B}}^2} = \frac{\abs{\bar{B}}^4 - 4 (B \cdot \bar{B})^2}{4 \abs{\bar{B}}^4} = \lambda \mu = \frac{\abs{\bar{u}}^4 - 4 (u \cdot \bar{u})^2}{4 \abs{\bar{u}}^4} = \frac{s^2-\abs{u}^2}{\abs{\bar{u}}^2}.
\end{equation}
Now \eqref{e: conclusion on bar b and bar u} immediately implies the first equality in \eqref{e: condition on bar u and bar b}, and the second one follows from \eqref{e:Apuyht5}:
\[\abs{\bar{B}}^2 = 4 \left( r^2-\abs{B}^2+ \frac{(B \cdot \bar{B})^2}{\abs{\bar{B}}^2} \right) = 4 (r^2 - \abs{B}^2 \sin^2 \alpha_{B,\bar{B}}).\]
Finally, \eqref{e: third condition on bar u and bar b} follows from \eqref{e:Apuyht2}, the equality $\lambda = 1/2 + u \cdot \bar{u}/\abs{\bar{u}}^2$ and \eqref{e: conclusion on bar b and bar u}:
\[\abs{B} \cos \alpha_{B,\bar{B}} = \frac{B \cdot \bar{B}}{\abs{\bar{B}}} = 
\frac{\abs{\bar{B}}}{\abs{\bar{u}}} \frac{u \cdot \bar{u}}{\abs{\bar{u}}} = \sqrt{\frac{r^2-\abs{B}^2}{s^2-\abs{u}^2}} \abs{u} \cos \alpha_{u,\bar{u}}.\]

Conversely, if \eqref{e: condition on bar u and bar b}--\eqref{e: second condition on bar u and bar b} hold, we intend to show that \eqref{e:Apuyht1} holds with the choice $\lambda = 1/2 + B \cdot \bar{B}/\abs{\bar{B}}^2$. First note that $\lambda \in [0,1]$ since \eqref{e: condition on bar u and bar b} yields
\[\frac{\abs{B \cdot \bar{B}}^2}{\abs{\bar{B}}^2} = \frac{\abs{B}^2 - \abs{B}^2 \sin^2 \alpha_{B,\bar{B}}}{\abs{\bar{B}}^2} \leq \frac{r^2 - \abs{B}^2 \sin^2 \alpha_{B,\bar{B}}}{\abs{\bar{B}}^2} = \frac{1}{4}.\]
Then \eqref{e:Apuyht5} holds, and similarly $\abs{B-\lambda \bar{B}} = r$ and $\abs{u + \mu \bar{u}} = \abs{u - \lambda \bar{u}} = s$. The identity \eqref{e: Form of e} is obtained by noting that \eqref{e: conclusion on bar b and bar u} yields
\[\lambda \mu = \frac{\sqrt{(r^2-\abs{B}^2)(s^2-\abs{u}^2)}}{\abs{\bar{B}} \abs{\bar{u}}}.\]
\end{proof}

Recall that our aim is to prove \eqref{e: hull}. When $\abs{B} \leq r$, $\abs{u} \leq s$ and $\bar{E} \in \bar{B}(0,1) \setminus \{0\}$ with $B \cdot \bar{E} = 0$, our aim is, therefore, to find $\bar{B}$ and $\bar{u}$ satisfying $\bar{B} \times \bar{u}/(\abs{\bar{B}} \abs{\bar{u}}) = \bar{E}$ along with \eqref{e: condition on bar u and bar b}--\eqref{e: second condition on bar u and bar b}.

\begin{proof}[Proof of \eqref{e: hull}]
Suppose first $B \neq 0$. Now $\bar{B},\bar{u} \neq 0$ satisfy $\bar{B} \times \bar{u}/[|\bar{B}| |\bar{u}|] = \bar{E}$ if and only if 
\begin{equation} \label{e:Geometric conditions}
\bar{B},\bar{u} \in \tp{span}\{B,B \times \bar{E}\}, \qquad \sin \alpha_{\bar{B},\bar{u}} = |\bar{E}|.
\end{equation}
We therefore consider pairs $(\bar{B},\bar{u})$ which satisfy \eqref{e:Geometric conditions}. Given a direction $\bar{B}/\abs{\bar{B}}$, note that $\bar{u}/\abs{\bar{u}}$ is uniquely fixed by \eqref{e:Geometric conditions} once we choose $\alpha_{\bar{B},\bar{u}} \in [0,2\pi)$ to be minimal. Also note that $\{B,\bar{E},B \times \bar{E}\}$ is an orthogonal basis of $\R^3$.

Our task is to show that \eqref{e: third condition on bar u and bar b} can be satisfied simultaneously with $\sin \alpha_{\bar{B},\bar{u}} = \abs{\bar{E}}$; then \eqref{e: condition on bar u and bar b} is obtained simply by scaling $\bar{B}$ and $\bar{u}$ and \eqref{e: second condition on bar u and bar b} holds automatically since $\bar{B},\bar{u} \in \tp{span}\{B,B \times \bar{E}\}$.

Aiming to solve \eqref{e: third condition on bar u and bar b}, we define $G \colon [0,2\pi) \to \R$ by
\[G(\alpha_{B,\bar{B}}) \defeq \abs{B} \cos \alpha_{B,\bar{B}} - \sqrt{\frac{r^2-\abs{B}^2}{s^2-\abs{u}^2}} \abs{u} \cos \alpha_{u,\bar{u}}.\]
Denoting $\alpha_{B,\bar{B}_1} = \pi/2$ and $\alpha_{B,\bar{B}_2} = 3\pi/2$, we have $\cos \alpha_{B,\bar{B}_1} = \cos \alpha_{B,\bar{B}_2} = 0$ and $\bar{u}_2/|\bar{u}_2| = - \bar{u}_1/|\bar{u}_1|$. Therefore, $G(\pi/2) = - G(3\pi/2)$. By Bolzano's theorem, $G(\alpha_{B,\bar{B}}) = 0$ for some $\alpha_{B,\bar{B}} \in [\pi/2,3\pi/2]$, that is, \eqref{e: third condition on bar u and bar b} holds.

We finish the proof by covering the case $B = 0$. We choose $\bar{u} \in \{u,\bar{E}\}^\perp \neq \{0\}$ and then select $\bar{B} \neq 0$ such that $\bar{B} \times \bar{u}/(|\bar{B}| \abs{\bar{u}}) = \bar{E}$. Again, \eqref{e: third condition on bar u and bar b} is satisfied, \eqref{e: condition on bar u and bar b} follows by scaling and \eqref{e: second condition on bar u and bar b} is immediate.
\end{proof}

\begin{proof}[Proof of Theorem \ref{Toinen päätulos}]
We have proved \eqref{e: hull} above, and the proof shows that \eqref{e: hull} also holds when $K_{r,s}^\Lambda$ is replaced by $K_{r,s}^{1,\Lambda}$. The proof of Theorem \ref{Toinen päätulos} is completed once we show that the set described in \eqref{e: hull} coincides with the relaxation $\widetilde{K_{r,s}}$.

For the proof of the fact that $\widetilde{K_{r,s}} \supset K_{r,s}^{1,\Lambda}$ see~\cite[pp. 162--163]{Tartar}. On the other hand, the inclusion $\widetilde{K_{r,s}} \subset K_{r,s}^\Lambda$ holds since $K_{r,s}^\Lambda = G_1^{-1}\{0\} \cap G_2^{-1}(-\infty,0]$ and the quadratic $\Lambda$-affine function $G_1$ and the convex function $G_2$ are lower semicontinuous on sequences of solutions of \eqref{eq:relaxedMHD} (see~\cite[Corollary 13]{Tartar}).
\end{proof}

\section{The relaxation of the stationary model under incompressibility} \label{s:The relaxation of the stationary model under incompressibility}
The $\Lambda$-convex hull $K_{r,s}^\Lambda$ remains unchanged if we incorporate the incompressibility condition $\nabla \cdot \u = 0$ or consider stationary solutions; in each case, the wave cone is given by Proposition \ref{prop:wavecone}. For completeness, we also compute the $\Lambda$-convex hull for the stationary kinematic dynamo equations under incompressibility. In this case, the wave cone, which we denote by $\Lambda_s$, is smaller.

\begin{proposition} \label{prop:wavecone s}
$\Lambda_s = \{(B,u,E) \colon B \cdot E = u \cdot E = 0\}$.
\end{proposition}

\begin{proof}
First suppose $B \cdot E = u \cdot E = 0$. Our aim is to find $\xi \in \R^3 \setminus \{0\}$ such that
\begin{equation} \label{eq:waveconeconditions}
B \cdot \xi = u \cdot \xi = 0, \qquad E \times \xi = 0.
\end{equation}
First, if $B \times u \neq 0$, then $\{B,u,B\times u\}$ is a basis of $\R^3$. Thus $B \cdot E = u \cdot E = 0$ implies $E = k \, B \times u$ for some $k \in \R$, so that we can choose $\xi = B \times u$. If $B \times u = 0$ and $E \neq 0$, we set $\xi = E$. If $B \times u = E = 0$, we choose any $\xi \in \{B,u\}^\perp \setminus \{0\}$.

Suppose then $(B,u,E) \in \Lambda$ so that \eqref{eq:waveconeconditions} holds for some $\xi \neq 0$. Now $E \times \xi = 0$ yields $E = k \xi$ for some $k \in \R$. Thus $B \cdot E = u \cdot E = 0$ by \eqref{eq:waveconeconditions}.
\end{proof}

The corresponding relaxation is characterised as follows:
\begin{theorem}\label{t:Stationaarinen puristumaton verho}
The relaxation, first laminate and $\Lambda_s$-convex hull of the stationary kinematic dynamo equations under incompressibility coincide and can be written as
\begin{align} \label{e: hull b}
K_{r,s}^{\Lambda_s}
= & \{(B,u,E) \colon \abs{B} \leq r, \; \abs{u} \leq s, \nonumber \\
& B \cdot E = 0, \; u \cdot E = 0, \; \abs{E-B \times u} \leq \sqrt{(r^2- \abs{B}^2) (s^2 - \abs{u}^2)}\}.
\end{align}
\end{theorem}

The proof of Theorem \ref{t:Stationaarinen puristumaton verho} is almost identical to that of Theorem \ref{Toinen päätulos}. Below we briefly indicate the necessary changes. First, the extra condition $u \cdot E = 0$ of the wave cone leads to a new $\Lambda$-affine function:
\begin{proposition}
The function $G_3(B,u,E) \defeq u \cdot E$ is $\Lambda_s$-affine and vanishes in $K_{r,s}^{\Lambda_s}$.
\end{proposition}

Lemmas \ref{cor:Lambda-affine functions}--\ref{prop:useful Lambda-convex function} clearly hold when $\Lambda$ is replaced by $\Lambda_s$. The inclusion "$\subset$" in \eqref{e: hull b} is then proved just like Lemma \ref{c: Upper bound on hull}. It remains to cover the inclusion "$\supset$" in \eqref{e: hull b} when $\Lambda$ is replaced by $\Lambda_s$.

First, Lemma \ref{l:Simple combinations} and its proof extend to $\Lambda_s$ verbatim. In Lemma \ref{l: lemma for main theorem} and its proof, the only necessary change is adding the condition $u \cdot \bar{B} \times \bar{u} = 0$ to formulas \eqref{e: second condition on bar u and bar b} and \eqref{e:Apuyht1}. The proof of \eqref{e: hull b} requires a bit more modification, as we indicate below.

\begin{proof}[Proof of \eqref{e: hull b}]
When $\abs{B} \leq r$, $\abs{u} \leq s$ and $\bar{E} \in \bar{B}(0,1) \setminus \{0\}$ with $B \cdot \bar{E} = u \cdot \bar{E} = 0$, our aim is to find $\bar{B}, \bar{u} \neq 0$ such that $\bar{B} \times \bar{u}/[|\bar{B}| |\bar{u}|] = \bar{E}$, \eqref{e: condition on bar u and bar b}--\eqref{e: second condition on bar u and bar b} hold and $u \cdot \bar{B} \times \bar{u} = 0$.

If $B \times u = 0$, we can follow the proof of \eqref{e: hull} verbatim. Suppose, therefore, that $B \times u \neq 0$. Now $\bar{B}, \bar{u}$ satisfy $\bar{B} \times \bar{u}/[|\bar{B}| |\bar{u}|] = \bar{E}$ if and only if
\begin{equation} \label{e:Geometric conditions b}
\bar{B},\bar{u} \in \textup{span}\{B,u\}, \qquad \sin \alpha_{\bar{B},\bar{u}} = |\bar{E}|.
\end{equation}
Again, we consider pairs $(\bar{B},\bar{u})$ which satisfy \eqref{e:Geometric conditions b} and need to show that \eqref{e: third condition on bar u and bar b} can be satisfied simultaneously with $\sin \alpha_{\bar{B},\bar{u}} = |\bar{E}|$. Just like in the proof of \eqref{e: hull}, the function $G \colon [0,2\pi) \to \R$, $G(\alpha_{B,\bar{B}}) \defeq \abs{B} \cos \alpha_{B,\bar{B}} - \sqrt{(r^2-\abs{B}^2)/(s^2-\abs{u}^2)} \abs{u} \cos \alpha_{u,\bar{u}}$ has a zero, which proves the claim.
\end{proof}

\bigskip
\footnotesize
\noindent\textit{Acknowledgments.}
We thank Daniel Faraco for useful discussions.

\bibliography{Dynamo}
\bibliographystyle{amsplain}
\end{document}